\documentclass[12pt,a4paper]{article}

\usepackage[a4paper,margin=2cm]{geometry}

\usepackage{amsmath}
\usepackage{amsthm}
\usepackage{subcaption}
\usepackage{amssymb}
\usepackage{enumitem}
\usepackage{graphicx}
\usepackage{nccmath}
\usepackage{setspace}
\usepackage[hyperfootnotes=false]{hyperref}
\usepackage{enumerate}
\usepackage{multirow}
\usepackage{float}
\usepackage{footnote}
\usepackage{titling}
\usepackage[dvipsnames]{xcolor}

\newtheorem{theorem}{Theorem}

\newtheorem{lemma}[theorem]{Lemma}
\newtheorem{conjecture}{Conjecture}

\thanksmarkseries{arabic}

\theoremstyle{definition}
\newtheorem{example}[theorem]{Example}
\newtheorem*{question}{Question}

\def \mod#1{{\:({\rm mod}\ #1)}}
\def \N{\mathbb{N}}

\def \leq {\leqslant}
\def \geq {\geqslant}

\let\oldproofname=\proofname
\renewcommand{\proofname}{\textup{\textbf{\oldproofname}}}

\title{On determining when small embeddings of partial Steiner triple systems exist}
\author{Darryn Bryant \thanks{School of Mathematics and Physics, The University of Queensland, QLD 4072, Australia} \qquad Ajani De Vas Gunasekara \thanks{
School of Mathematical Sciences, Monash University, Victoria 3800, Australia} \qquad  Daniel Horsley \thanksmark{2} }
\date{}

\begin{document}

\maketitle

\begin{abstract}
A \emph{partial Steiner triple system} of order $u$ is a pair $(U,\mathcal{A})$ where $U$ is a set of $u$ elements and $\mathcal{A}$ is a set of triples of elements of $U$ such that any two elements of $U$ occur together in at most one triple. If each pair of elements occur together in exactly one triple it is a \emph{Steiner triple system}. An \emph{embedding} of a partial Steiner triple system $(U,\mathcal{A})$ is a (complete) Steiner triple system $(V,\mathcal{B})$ such that $U \subseteq V$ and $\mathcal{A} \subseteq \mathcal{B}$. For a given partial Steiner triple system of order $u$ it is known that an embedding of order $v \geq 2u+1$ exists whenever $v$ satisfies the obvious necessary conditions. Determining whether ``small'' embeddings of order $v < 2u+1$ exist is a more difficult task. Here we extend a result of Colbourn on the $\mathsf{NP}$-completeness of these problems. We also exhibit a family of counterexamples to a conjecture concerning when small embeddings exist.
\end{abstract}

\section{Introduction}

A \textit{partial Steiner triple system of order $u$}, or PSTS$(u)$, is a pair $(U,\mathcal{A})$ where $U$ is a set of $u$ elements and $\mathcal{A}$ is a set of triples of elements of $U$ with the property that any two elements of $U$ occur together in at most one triple. If any two elements of $U$ occur together in exactly one triple then $(U,\mathcal{A})$ is a \textit{Steiner triple system of order $u$}, or STS$(u)$. It is well known that a Steiner triple system of order $u$ exists if and only if $u \equiv 1,3 \mod{6}$ \cite{ColRos}. This was first proved by Kirkman in \cite{kirkman 1847}. We call integers congruent to 1 or 3 modulo 6 \emph{admissible} and denote the set of positive admissible integers by $\N^\dag$.

A \textit{$K_3$-decomposition} of a graph $G$ is a set of triangles in $G$ such that each edge of $G$ is in exactly one triangle in the set. A Steiner triple system of order $v$ is equivalent to a $K_3$-decomposition of $K_v$ and a partial Steiner triple system of order $u$ is equivalent to a $K_3$-decomposition of some subgraph of $K_u$. The \textit{leave} of a partial Steiner triple system $(U,\mathcal{A})$  is the graph $L$ having vertex set $U$ and the edge set $E(L) = \{ xy : \hbox{$\{x,y,z\} \notin \mathcal{A}$ for all $z \in U$}\}$.  For a partial Steiner triple system $(U,\mathcal{A})$, we say that a (complete) Steiner triple system $(V,\mathcal{B})$ is an \textit{embedding} of $(U,\mathcal{A})$ if $U \subseteq V$ and $\mathcal{A} \subseteq \mathcal{B}$. A (proper) \textit{$c$-edge colouring} of a graph $G$ is an assignment of colours, chosen from some set of $c$ colours, to the edges of $G$ in such a way that any two edges incident with the same vertex receive distinct colours. All edge colourings considered in this paper will be proper.

It is known that any PSTS$(u)$ has an embedding of order $v$ for each admissible integer $v \geq 2u +1$ \cite{Lindner's proof}. Moreover, the bound of $v \geq 2u +1$ cannot be improved in general due to the fact that for any $u \geq 9$ there exists a PSTS$(u)$ which cannot be embedded in an STS$(v)$ for any $v < 2u+1$ \cite{ColRos}. Of course, many partial Steiner triple systems do have embeddings of order less than $2u+1$. We call such embeddings \textit{small} embeddings.

This paper concerns the problem of determining whether a given partial Steiner triple system has a small embedding of a specified order. Various aspects of this problem have been addressed in many papers (see \cite{Bryant2002, Lindner's proof, colbourn paper, Hors2014, NenadovSudakovwanger2019} for example). In this paper we provide updates on two of these contributions, namely \cite{colbourn paper} and \cite{Bryant2002}.

In \cite{colbourn paper} Colbourn showed the problem of determining whether a given partial Steiner triple system has a small embedding is $\mathsf{NP}$-complete. In order to be more precise we make some definitions. Call a function $F: \N \rightarrow \mathcal{P}(\N)$ \emph{admissible} if $F(u) \subseteq \{x \in \N^\dag: x \geq u\}$ for each $u \in \N$. For each admissible function $F$ we define a decision problem as follows.
\begin{description}[itemsep=0mm,parsep=0mm,topsep=1mm]
    \item[\textmd{\textsc{$F$-embed}}]

    \item[\textit{\textmd{Instance:}}]
A partial Steiner triple system $(U,\mathcal{A})$.

    \item[\textit{\textmd{Question:}}]
Does $(U,\mathcal{A})$ have an embedding of order $v$ for some $v \in F(|U|)$?
\end{description}

More formally, Colbourn's result in \cite{colbourn paper} is that \textsc{$F^*$-embed} is $\mathsf{NP}$-complete, where $F^*$ is the admissible function defined by $F^*(u)=\{x \in \N^\dag: u \leq x < 2u+1\}$ for each $u \in \N$. Here we extend this result by proving the following theorem. For a subset $S$ of $\N$, we say that integers in $S$ \emph{occur polynomially often} if there is a polynomial $P(x)$ such that, for each $n \in \N$, we have $\{s \in S:n \leq s \leq P(n)\} \neq \emptyset$.

\begin{theorem}\label{T:FembedNPComplete}
Let $F$ be an admissible function. The decision problem \textup{\textsc{$F$-embed}} is $\mathsf{NP}$-complete if there exists a real number $\epsilon>0$ such that integers $u$ for which $F(u)\neq \emptyset$ and $\max(F(u))<(2-\epsilon)u$ occur polynomially often.
\end{theorem}

Note that the answer to \textsc{$F$-embed} for any PSTS$(u)$ is obviously negative if $F(u)=\emptyset$ and is affirmative if $\max(F(u)) \geq 2u$ (and hence at least $2u+1)$ because embeddings are known to exist for all non-small admissible orders. Thus, Theorem~\ref{T:FembedNPComplete} is best possible except for the $\epsilon$ term and the mild condition of being nontrivial polynomially often. This latter condition merely rules out choices of $F$ that are pathological in the sense that there are exponentially long intervals of orders $u$ for which \textsc{$F$-embed} is trivial for all Steiner triple systems of order $u$.

For vertex-disjoint graphs $G$ and $H$, we let $G \vee H$ denote the graph with vertex set $V(G) \cup V(H)$ and edge set $E(G) \cup E(H) \cup \{xy:x \in V(G),y\in V(H)\}$. In \cite{Bryant2002} the first author made a conjecture about the existence of $K_3$-decompositions of $L \vee K_w$.
It is obvious that a partial Steiner triple system of order $u$ with a leave $L$ can be embedded in a Steiner triple system of order $v = u+w$ if and only if there exists a $K_3$-decomposition of $L \vee K_w$. He conjectured that certain conditions that can be seen to be necessary for the existence of a $K_3$-decomposition of $L \vee K_w$ are also sufficient (see \cite[Lemma 2.1]{Bryant2002} for a proof of their necessity). For graphs $G$ and $H$ we define $G-H$ to be the graph with vertex set $V(G)$ and edge set $E(G) \setminus E(H)$.

\begin{conjecture} [\cite{Bryant2002}]\label{C:bryant}
Let $L$ be a graph with $u$ vertices and let $w$ be a nonnegative integer. Then there exists a $K_3$-decomposition of $L \vee K_w$ if and only if the following conditions are satisfied.
\begin{itemize}
  \item[\textup{(1)}] $\deg_L(x) \equiv w \mod{2}$ for each vertex $x$ of $L$;
  \item[\textup{(2)}] $u+w$ is odd for $w>0$;
  \item[\textup{(3)}] $|E(L)| + uw + \binom{w}{2} \equiv 0 \mod{3}$; and
  \item[\textup{(4)}] There exists a subgraph $G$ of $L$ such that
  \begin{itemize}
      \item[\textup{(i)}] $L - G$ has a $K_3$-decomposition;
      \item[\textup{(ii)}] $w^2 - (u+1)w + 2|E(G)| \geq 0$;
      \item[\textup{(iii)}] $G$ is $w$-edge colourable.
  \end{itemize}
\end{itemize}
\end{conjecture}

Theorem~\ref{T:FembedNPComplete} and the main result of \cite{colbourn paper} suggest that there may be no efficient algorithm for determining which small orders a partial Steiner triple system has an embedding into. But Conjecture~\ref{C:bryant} postulates a neat characterization of these orders in terms of chromatic indices of graphs. Here we suggest that things may not be so simple by exhibiting a family of counterexamples to Conjecture~\ref{C:bryant}.

\begin{theorem}\label{T:counterExample}
For each even integer $w \geq 4$, there is a partial Steiner triple system whose leave is a counterexample to Conjecture~$\ref{C:bryant}$.
\end{theorem}

For $w=4$ we explicitly exhibit a system of order 15 whose leave is a counterexample to Conjecture~\ref{C:bryant} (see Example~\ref{E:darryn's example}). For $w \geq 6$, however, we merely establish the existence of appropriate systems with large (unspecified) orders.

\section{Hardness of finding small embeddings of specified orders}

There are sensible questions about small embeddings that Colbourn's result \cite{colbourn paper} does not cover. For example we could ask: does a given partial Steiner triple system of order $u$ have an embedding of order $u+10$? Similarly, we could ask: does a given partial Steiner triple system of order $u$ have an embedding of order between $\frac{3u}{2}$ and $\frac{5u}{3}$? Colbourn's result does not say whether either of these questions is $\mathsf{NP}$-complete. Theorem~\ref{T:FembedNPComplete} shows that many questions of this kind are also hard.

We aim to prove Theorem~\ref{T:FembedNPComplete} by reducing to \textsc{$F$-embed} from the problem of whether a cubic graph is properly $3$-edge colourable, which is well known to be $\mathsf{NP}$-complete \cite{Hol}. Critical to this approach will be the construction of a class of partial Steiner triple systems which we now define. For positive integers $u$ and $v$ and a cubic graph $G$, a \emph{$(u,v,G)$-background} is a PSTS$(u)$ that has no embedding of order less than $v$ and, further, has an embedding of order $v$ if and only if $G$ is $3$-edge colourable.

\begin{lemma} \label{big lemma 2}
If $G$ is a cubic graph of order $n \geq 74$ and $u$ and $v$ are integers such that $v \equiv 1,3 \mod{6}$, $u \geq 4n+43$ and $u \leq v \leq 2u-2n-13$, then there exists a $(u,v,G)$-background.
\end{lemma}

Before proceeding to prove Lemma~\ref{big lemma 2}, we show how Theorem~\ref{T:FembedNPComplete} can be proved from Lemma~\ref{big lemma 2}.

\begin{proof}[\textup{\textbf{Proof that  Lemma~\ref{big lemma 2} implies Theorem~\ref{T:FembedNPComplete}.}}]
Let $F$ be an admissible function and let $\epsilon>0$ be a real number such that integers $u$ for which $F(u)\neq \emptyset$ and $\max(F(u))<(2-\epsilon)u$ occur polynomially often.
We reduce to \textsc{$F$-embed} from the problem of whether a cubic graph is $3$-edge colourable, which is well known to be $\mathsf{NP}$-complete \cite{Hol}. Of course, this latter problem remains $\mathsf{NP}$-complete if we exclude finitely many inputs by requiring that the graph have order at least $74$.

Suppose we are given a cubic graph $G$ of order $n \geq 74$. By the properties of $F$ and $\epsilon$ we can choose an integer $u$ such that $u \geq \max(4n+43,\frac{1}{\epsilon}(2n+13))$, $u$ is polynomial in $n$ for some polynomial independent of $n$, $F(u) \neq \emptyset$
and $\max(F(u))<(2-\epsilon)u$. Then, because $F$ is admissible and $u \geq \frac{1}{\epsilon}(2n+13)$, we have $u \leq \max(F(u)) \leq 2u-2n-13$. Let $v=\max(F(u))$. Thus $u$ and $v$ satisfy the hypothesis of Lemma \ref{big lemma 2} and hence there exists a $(u,v,G)$-background $(U,\mathcal{A})$. Because $(U,\mathcal{A})$ is a $(u,v,G)$-background, the answer to \textsc{$F$-embed} for input $(U,\mathcal{A})$ will be affirmative if and only if $G$ is $3$-edge colourable.
\end{proof}

So our goal in the rest of this section will be to prove Lemma~\ref{big lemma 2}. We introduce some further notation that we will require. For graphs $G$ and $H$ we define $G \cup H$ to be the graph with vertex set $V(G) \cup V(H)$ and edge set $E(G) \cup E(H)$.
For a set $S$ we denote the complete graph with vertex set $S$ by $K_S$ and denote its complement, the graph with vertex set $S$ and empty edge set, by $\overline{K_S}$. For disjoint sets $S$ and $T$, we denote the complete bipartite graph with parts $S$ and $T$ by $K_{S,T}$. We say a graph is \emph{even} if each of its vertices has even degree. A \emph{$K_3$-packing} of a graph $G$ is a $K_3$-decomposition of some subgraph $H$ of $G$ and the \emph{leave} of such a packing is the graph $G-H$. It will be useful for us to blur the distinction between partial Steiner triple systems and $K_3$-packings by representing the latter as sets of vertex triples rather than as sets of triangles. We do this throughout the paper.

\begin{lemma}\label{L:colouringToDecomp}
Let $G$ be a cubic graph and let $Z$ be a vertex set such that $|Z|=3$ and $Z$ is disjoint from $V(G)$.
\begin{itemize}
    \item[\textup{(i)}]
If $G$ is $3$-edge colourable then there is a $K_3$-decomposition of $\overline{K_Z} \vee G$.
    \item[\textup{(ii)}]
If $G$ is not $3$-edge colourable then the leave of any $K_3$-packing of $\overline{K_Z} \vee G$ contains an edge incident with a vertex in $Z$.
\end{itemize}
\end{lemma}

\begin{proof} Let $n$ be the order of $G$.
\begin{itemize}
    \item[(i)]
Assume $G$ is $3$-edge colourable. Let $\gamma$ be a proper $3$-edge colouring of $G$ with colour set $Z$. Then
\[\mathcal{P}=\{\{x,y,\gamma(xy)\}:xy \in E(G)\}\]
is a $K_3$-decomposition of $\overline{K_Z} \vee G$. Each edge of $G$ is obviously in exactly one triangle in $\mathcal{P}$, and the fact that $\gamma$ is a proper $3$-edge colouring implies that each edge in $K_{Z,V(G)}$ is in exactly one triangle in $\mathcal{P}$.
     \item[(ii)]
Suppose for a contradiction that $G$ is not $3$-edge colourable and that there is a triangle packing $\mathcal{P}$ of $\overline{K_Z} \vee G$ such that every edge incident with a vertex in $Z$ is in some triangle of $\mathcal{P}$. Then each vertex in $Z$ is in $\frac{n}{2}$ triangles in $\mathcal{P}$ and hence for every edge $xy$ in $E(G)$ there is a triangle $\{x,y,z\}$ in $\mathcal{P}$ for some $z \in Z$. Define an edge colouring $\gamma$ of $G$ with colour set $Z$ by setting $\gamma(xy)=z$ for each $xy \in E(G)$, where $z$ is the unique element of $Z$ such that $\{x,y,z\} \in \mathcal{P}$. Then $\gamma$ is a proper $3$-edge colouring of $G$, which is a contradiction. \qedhere
\end{itemize}
\end{proof}

Lemma~\ref{lemma222} is our first step toward constructing $(u,v,G)$-backgrounds.

\begin{lemma}
\label{lemma222}
Let $G$ be a cubic graph of order $n \geq 74$. Let $A$ be a vertex set such that $V(G) \subseteq A$, $|A| \geq 2n+1$ and $|A| \equiv 1,3 \mod{6}$, and let $Z \subseteq A \setminus V(G)$ such that $|Z|=3$. Then there exists a partial Steiner triple system $(A,\mathcal{B}_0)$ whose leave $L$ has edge set $E(\overline{K_Z} \vee G)$.
\end{lemma}

\begin{proof}
By \cite[Theorem 5.2]{Hors2014}, if $G'$ is an even graph of order $a$ such that $a \geq 103$, $|E(G')| \equiv  0 \mod{3}$, $|E(G')| \geq \binom{a}{2} - \frac{1}{128}(3a^2 - 54a - 409)$ and at least $\frac{1}{8}(3a+17)$ vertices of $G'$ have degree $a-1$, then there is a $K_3$-decomposition of $G'$.

Let $a=|A|$ and let $G'=K_A-(\overline{K_Z} \vee G)$. We will complete the proof by showing that $G'$ satisfies the conditions above. Note that $K_A$ is even because $a \equiv 1,3 \mod{6}$ and $\overline{K_Z} \vee G$ is even because $n$ is even, so $G'$ is even. Next, we have $|V(G')|=a > 103$ because $a \geq 2n+1$ and $n \geq 74$. Now
$|E(G')|= \binom{a}{2} - (3n + \frac{3n}{2}) = \binom{a}{2} - \frac{9n}{2}$ and hence $|E(G')| \equiv  0 \mod{3}$ because $a \equiv 1,3 \mod{6}$. Also $|E(G')| \geq \binom{a}{2} - \frac{1}{128}(3a^2 - 54a - 409)$ because
\[\tfrac{1}{128}(3a^2 - 54a - 409) \geq \tfrac{1}{128}((6n-51)(2n+1) - 409) \geq \tfrac{393}{128}(2n+1)-\tfrac{409}{128} > \tfrac{9n}{2}\]
where the first inequality holds because $a \geq 2n+1$, and the second and third hold because $n \geq 74$. Finally, $a-n-3$ vertices in $G'$ have degree $a-1$ and $a-n-3 > \frac{1}{8}(3a+17)$ because $5a \geq 10n+5 > 8n+41$ where the first inequality holds because $a \geq 2n+1$ and the second holds because $n \geq 74$.
\end{proof}

We are now able to construct some of the $(u,v,G)$-backgrounds we require. We do this in Lemma~\ref{L:createCandidate} and then prove that they are in fact $(u,v,G)$-backgrounds in Lemma~\ref{big lemma}.

\begin{lemma}
\label{L:createCandidate}
Let $G$ be a cubic graph of order $n \geq 74$ and let $u$ and $d$ be integers such that $d \geq n + 2$, $u \geq d+2n+3$, $d \equiv 0 \mod{6}$ and $u \equiv 1,3 \mod{6}$. There exists a PSTS$(u)$ $(U,\mathcal{A})$ whose leave has edge set $E((\overline{K_Z} \vee G) \cup K_{A',D})$ where
\begin{itemize}
    \item
$\{A',D,V(G) \cup \{x\}\}$ is a partition of $U$ for some $x \in U \setminus V(G)$;
    \item
$|D|=d$;
    \item
$Z \subseteq A'$ with $|Z|=3$.
\end{itemize}

\end{lemma}

\begin{proof}
Let $U$ be a set with $|U|=u$ and $V(G) \subseteq U$, and let $\{A',D,V(G) \cup \{x\}\}$ be a partition of $U$ satisfying the conditions of the lemma. Let $A''=V(G) \cup \{x\}$ and let $A = U \setminus D = A'\cup A''$.

Observe that $|A|=u-d \geq 2n+3$ and $|A| \equiv 1,3 \mod{6}$ because $u \equiv 1,3 \mod{6}$ and $d \equiv 0 \mod{6}$. Thus by Lemma~\ref{lemma222} there exists a partial Steiner triple system $(A,\mathcal{B}_0)$ whose leave has edge set $E(\overline{K_Z} \vee G)$. If there exists a $K_3$-decomposition $\mathcal{B}_1$ of $K_{A'' \cup D}-K_{A''}$, then $(U,\mathcal{B}_0 \cup \mathcal{B}_1)$ will indeed be a partial Steiner triple system whose leave has edge set $E((\overline{K_Z} \vee G) \cup K_{A',D})$ and we are finished. So it suffices to show that such a $K_3$-decomposition exists.

It is known (see \cite{DoyWil,MenRos}) that a $K_3$-decomposition of $K_v - K_w$ exists if and only if $v$ and $w$ are odd, $v \geq 2w+1$, and $\binom{v}{2} - \binom{w}{2} \equiv 0 \mod{3}$. Now $|A'' \cup D|=n+1+d$ and $|A''|=n+1$ are both odd because $n$ and $d$ are even, and $n+1+d \geq 2n+3$ because $d \geq n+2$. Finally, $\binom{d + n+1}{2} - \binom{n+1}{2} = \frac{1}{2}d(d + 2n +1)\equiv 0 \mod{3}$ because $d\equiv 0 \mod{6}$.
\end{proof}

\begin{lemma} \label{big lemma}
Let $G$ be a cubic graph of order $n \geq 74$ and let $u$ and $v$ be integers such that $u \geq 3n+5$, $\frac{1}{2}(3u-n-1) \leq v \leq 2u-2n-3$ and $u \equiv v \equiv 1,3 \mod{6}$. Then there exists a $(u,v,G)$-background.
\end{lemma}

\begin{proof}
Let $d=v-u$ and note that $d \equiv 0 \mod{6}$ because $u \equiv v \mod{6}$. Let $(U,\mathcal{A})$ be a PSTS$(u)$ whose leave has the edge set $E((\overline{K_Z} \vee G) \cup K_{A',D})$ where
\begin{itemize}
    \item
$\{A',D,V(G) \cup \{x\}\}$ is a partition of $U$ for some $x \in U \setminus V(G)$;
    \item
$|D|=d$;
    \item
$Z \subseteq A'$ with $|Z|=3$.
\end{itemize}
The existence of such a partial Steiner triple  system has been proved in Lemma \ref{L:createCandidate}, noting that $v \leq 2u-2n-3$ implies that $u \geq d+2n+3$ and that $u \geq 3n+5$ and $\frac{1}{2}(3u-n-1) \leq v$ imply $d \geq n+2$. Let $L$ be the leave of $(U,\mathcal{A})$. We claim that $(U,\mathcal{A})$ is a $(u,v,G)$-background.

We will first show that $(U,\mathcal{A})$ has no embedding of order less than $v=u+d$, and has no embedding of order $v=u+d$ if $G$ is not $3$-edge colourable. Suppose $(U,\mathcal{A})$ has an embedding $(U \cup W, \mathcal{A} \cup \mathcal{A}' \cup \mathcal{A}'')$ where $W$ is disjoint from $U$, triples in $\mathcal{A}'$ are subsets of $U$ and triples in $\mathcal{A}''$ each contain at least one vertex in $W$. Let $L'$ be the leave of $(U, \mathcal{A} \cup \mathcal{A}')$. We show that $|W| \geq d$ and that $|W| \geq d+1$ if $G$ is not $3$-edge colourable.

Consider any vertex $y \in A' \setminus Z$. Because the subgraph of $L$ induced by $D$ is empty, no triple in $\mathcal{A}'$ can contain $y$ and hence $\deg_{L'}(y)=\deg_L(y) = d$.  Each of the $d$ edges incident in $L'$ with $y$  is in a triple of $\mathcal{A''}$ whose third vertex is in $W$, and no two of these vertices in $W$ may be the same. Therefore, $|W| \geq d$.

Now further assume $G$ is not $3$-edge colourable. The triples in $\mathcal{A}'$ form a $K_3$-packing of $\overline{K_Z} \vee G$, so by Lemma~\ref{L:colouringToDecomp}(ii) there exists a vertex $z \in Z$ such that $\deg_{L'}(z) \geq d+1$. Each of the at least $d+1$ edges incident in $L'$ with $z$ is in a triple of $\mathcal{A''}$ whose third vertex is in $W$, and no two of these vertices in $W$ may be the same. Hence, $|W| \geq d+1$ if $G$ is not $3$-edge colourable.

Now, we will show that if $G$ is $3$-edge colourable then $(U,\mathcal{A})$ has an embedding of order $u+d$. Assume $G$ is $3$-edge colourable and let $V$ be a vertex set with $|V|=u+d$ and $U \subseteq V$. Let $A =  U \setminus D$ and let $a=|A|= u-d$. By Lemma~\ref{L:colouringToDecomp}(i), there is a $K_3$-decomposition $\mathcal{A}^\dag$ of $\overline{K_Z} \vee G$. Then $(U, \mathcal{A} \cup \mathcal{A}^\dag)$ is a PSTS$(u)$ whose leave has edge set $E(K_{A',D})$. Equivalently, $\mathcal{A} \cup \mathcal{A}^\dag$ is a $K_3$-decomposition of $K_{A} \cup K_{D \cup V(G) \cup \{x\}}$. Let $B=D \cup V(G) \cup \{x\}$.  It suffices to show that there is a $K_3$-decomposition $\mathcal{A}^\ddag$ of $K_V-(K_{A} \cup K_{B})$ because then $(V,\mathcal{A} \cup \mathcal{A}^\dag \cup \mathcal{A}^\ddag)$ will be an embedding of $(U,\mathcal{A})$ of order $u+d$.

By \cite[Theorem $3.1$]{ColOraRee}, there exists a $K_3$-decomposition of $K_{V}-(K_{A} \cup K_{B})$ if
\begin{itemize}[itemsep=0mm]
    \item[(i)]
$|B| \geq |A|$;
    \item[(ii)]
$|V|=2|B|+|A|-2|A \cap B|$;
    \item[(iii)]
$|A|$ and $|B|$ are odd;
    \item[(iv)]
$|A| \geq 2|A \cap B|+1$; and
    \item[(v)]
$(|B|-|A \cap B|)(|A| -2|A \cap B|-1) \equiv 0 \mod{3}$.
\end{itemize}
So it suffices to show that (i) -- (v) hold. Note $|V|=u+d=a+2d$, $|A|=a$, $|B|=d+n+1$, and $|A \cap B|=n+1$. Because $\frac{1}{2}(3u-n-1) \leq v$ we have that (i) holds, noting that $|B|=v-u+n+1$ and $|A|=u-d=2u-v$. Furthermore, (ii) and (iii) obviously hold, (iv) holds because $v \leq 2u-2n-3$, and (v) holds because $d \equiv 0 \mod{6}$. So there is a $K_3$-decomposition of $K_V-(K_{A} \cup K_{B})$ and the proof is complete.
\end{proof}

We can now obtain all the $(u,v,G)$-backgrounds we require by simply adding new vertices to those we have already constructed.

\begin{proof}[\textup{\textbf{Proof of Lemma~\ref{big lemma 2}}}]
Let $u$ and $v$ be integers satisfying the hypotheses of the lemma. If, for some integer $u' \leq u$, we can find a $(u',v,G)$-background $(U',\mathcal{A})$, then the partial Steiner triple system $(U,\mathcal{A})$ obtained from $(U',\mathcal{A})$ by adding $u-u'$ new vertices will be a $(u,v,G)$-background. So it suffices to find an integer $u' \leq u$ such that $u'$, $v$ and $G$ satisfy the hypotheses of Lemma \ref{big lemma}. We choose $u'$ to be the largest integer such that $u' \leq \min(u,\frac{1}{3}(2v+n+1))$ and $u' \equiv v \mod{6}$. This implies $u'$ must be odd.

\smallskip\noindent\textbf{Case 1.} If $u \leq \frac{1}{3}(2v+n+1)$, then $u-5 \leq u' \leq u$. Thus $\frac{1}{2}(3u'-n-1) \leq v$, because $u \leq \frac{1}{3}(2v+n+1)$ by the conditions of this case and $u' \leq u$. Also, $v \leq 2u'-2n-3$ because $v \leq 2u-2n-13$ and $u \leq u'+5$. Moreover $u' \geq 3n+5$ because $u \geq 4n+43$
and $u \leq u'+5$. So $u'$, $v$ and $G$ satisfy the hypotheses of Lemma \ref{big lemma}.

\smallskip\noindent\textbf{Case 2.} If $u > \frac{1}{3}(2v+n+1)$, then $\frac{1}{3}(2v+n+1)-6 < u' \leq \frac{1}{3}(2v+n+1)$. Thus $\frac{1}{2}(3u'-n-1) \leq v$ because $u' \leq \frac{1}{3}(2v+n+1)$. Also, $u' \geq 3n+5$ because $u' > \frac{1}{3}(2v+n+1)-6$ and $v \geq 4n+43$ imply that $u' > 3n + 23$. Finally $v \leq 2u'-2n-3$ because
\[2v<3u'-n+17<4u'-4n-6\]
where the first inequality holds because $\frac{1}{3}(2v+n+1)-6 < u'$ and the second holds because we have just seen that $u' > 3n+23$. So, again, $u'$, $v$ and $G$ satisfy the hypotheses of Lemma~\ref{big lemma}.
\end{proof}

We conclude this section by posing a natural question that is not answered by Theorem~\ref{T:FembedNPComplete}. Removing the $\epsilon$ term from Theorem~\ref{T:FembedNPComplete} would necessarily entail answering this question.

\begin{question}
Let $F'$ be the admissible function defined by $F'(u)=\{2u-1\}\cap\N^\dag$. Is \textsc{$F'$-embed} $\mathsf{NP}$-complete?
\end{question}

\section{Counterexamples to Conjecture~\ref{C:bryant}}

In this section we prove Theorem~\ref{T:counterExample} by exhibiting, for each even $w \geq 4$, a Steiner triple system whose leave is a counterexample to Conjecture~\ref{C:bryant}. We introduce some more notation. The maximum degree and minimum degree of a graph $G$ are denoted by $\Delta(G)$ and $\delta(G)$ respectively. The smallest number of colours required to edge colour a graph $G$ is the \emph{chromatic index} of $G$, denoted $\chi'(G)$. Vizing's theorem \cite{Viz} states that $\chi'(G) \in \{\Delta(G),\Delta(G)+1\}$ for any graph $G$ and K{\"o}nig's theorem \cite{Kon} states that $\chi'(G) = \Delta(G)$ for any bipartite graph $G$. A \emph{matching} is a 1-regular graph. Note that the edges assigned a particular colour by an edge colouring always induce a matching. In an edge-colouring of a graph we say that a colour $c$ \emph{hits} a vertex $u$ if there is an edge of colour $c$ incident with $u$. Otherwise we say $c$ \emph{misses} $u$.

Our first lemma in this section encapsulates our strategy for finding graphs that form counterexamples to Conjecture~\ref{C:bryant}.

\begin{lemma}\label{L:counterexampleIdea}
An even graph $L$ of odd order $u$ is a counterexample to Conjecture~\ref{C:bryant} for a given even integer $w$ if it satisfies
\begin{itemize}
    \item[\textup{(i)}]
$|E(L)| = \frac{1}{2}w(u-w+1)$;
    \item[\textup{(ii)}]
$\chi'(L)=w$;
    \item[\textup{(iii)}]
there are two vertices $d_1$ and $d_2$ of $L$ such that, in any $w$-edge colouring of $L$, the set of colours that hit $d_1$ equals the set of colours that hit $d_2$.
\end{itemize}
\end{lemma}

\begin{proof}
Let $L$ be an even graph of odd order $u$ that satisfies (i), (ii) and (iii) for a given $w$. We first prove that $L$ satisfies the conditions in Conjecture~\ref{C:bryant}. Obviously (1) and (2) of Conjecture~\ref{C:bryant} hold because $L$ is an even graph, $w$ is even and $u$ is odd. Also, (3) of Conjecture~\ref{C:bryant} holds because $|E(L)|+uw+ \binom{w}{2} = \frac{3}{2}uw \equiv 0 \mod{3}$ since $|E(L)| = \frac{1}{2}w(u-w+1)$. Moreover $|E(L)| = \frac{1}{2}w(u-w+1)$ implies $w^2 -(u+1)w+ 2|E(L)| = 0$ and so (4) of Conjecture~\ref{C:bryant} holds with $G=L$, noting that $\chi'(L)=w$. Hence $L$ satisfies all the conditions in Conjecture~\ref{C:bryant}.

Now let $W=\{1,\ldots,w\}$ be a set disjoint from $V(L)$ and suppose for a contradiction that $L \vee K_W$ has a $K_3$-decomposition $\mathcal{D}$. Call the edges of $L \vee K_W$ with one endpoint in $V(L)$ and one endpoint in $W$ \emph{cross edges} and call the other edges \emph{pure edges}. For $i \in \{0,1,2,3\}$, call triangles in $\mathcal{D}$ that contain exactly $i$ vertices in $V(L)$ \emph{type-$i$ triangles}. Now $L \vee K_W$ has $uw$ cross edges and $|E(L)|+\binom{w}{2}=\frac{1}{2}uw$ pure edges. Thus, because each triangle in $\mathcal{D}$ contains at most two cross edges and at least one pure edge, $\mathcal{D}$ must consist of $|E(L)|$ type-2 triangles and $\binom{w}{2}$ type-1 triangles.

The $|E(L)|$ type-2 triangles in $\mathcal{D}$ induce a proper edge colouring $\gamma$ of $L$ with the colour set $W$ defined by $\gamma(xy)=z$ for each $xy \in E(L)$, where $z$ is the unique element of $W$ such that $\{x,y,z\}$ is in $\mathcal{D}$. By (iii), in $\gamma$, the set of colours that hit $d_1$ equals the set of colours that hit $d_2$. Without loss of generality assume the set of colours that hit $d_1$ and $d_2$ is $\{3,4, \ldots w\}$ and so colours $1$ and $2$ miss $d_1$ and $d_2$. Thus the only edges incident with $d_1$ and $d_2$ that do not occur in type-2 triangles in $\mathcal{D}$ are $\{d_1,1\}$, $\{d_1,2\}$, $\{d_2,1\}$, $\{d_2,2\}$. So these must occur in type-1 triangles in $\mathcal{D}$. However, this implies the contradiction that both the triangles $\{1,2,d_1\}$ and $\{1,2,d_2\}$ occur in $\mathcal{D}$. Therefore $L$ is indeed a counterexample to Conjecture~\ref{C:bryant} for the given value of $w$.
\end{proof}

We first exhibit a PSTS($15$) whose leave forms a counterexample to Conjecture~\ref{C:bryant} for $w=4$.

\begin{example} \label{E:darryn's example}
Let $U=\{1,2,\dots,15\}$ and let $\mathcal{A}$ be the set consisting
of the following $27$ triples.
$$\begin{array}{llllll}
\{1,2,7\}&
\{1,3,12\}&
\{1,4,11\}&
\{1,8,15\}&
\{1,9,10\}&
\{1,13,14\}\\[2pt]
\{2,5,10\}&
\{2,6,13\}&
\{2,8,11\}&
\{2,9,14\}&
\{2,12,15\}&
\{3,7,8\}\\[2pt]
\{3,9,15\}&
\{3,10,14\}&
\{3,11,13\}&
\{4,7,15\}&
\{4,8,14\}&
\{4,9,13\}\\[2pt]
\{4,10,12\}&
\{5,7,13\}&
\{5,8,12\}&
\{5,9,11\}&
\{5,14,15\}&
\{6,7,10\}\\[2pt]
\{6,8,9\}&
\{6,11,15\}&
\{6,12,14\}&&&
\end{array}$$
Then $(U,A)$ is a PSTS($15$) and the leave $L$ of $(U,A)$ has two
components as shown in Figure~\ref{F:darrynLeave}.

We note that $|E(L)|=24$ and that $\chi'(L)=4$ because $\Delta(L)=4$ and a $4$-edge colouring of $L$ is given by the different line styles in Figure~\ref{F:darrynLeave}.
Further, in any $4$-edge colouring of $L$, it is
not difficult to see that the set of colours that hit vertex $1$ equals the set of colours that hit vertex $2$ (for a formal proof of this see Lemma~\ref{L:L1 colouring}). Thus $L$ satisfies the conditions of Lemma~\ref{L:counterexampleIdea} for $w=4$ and so is a counterexample to Conjecture~\ref{C:bryant} for $w=4$.

\begin{figure}[H]
\centering
\includegraphics[width=0.7\linewidth]{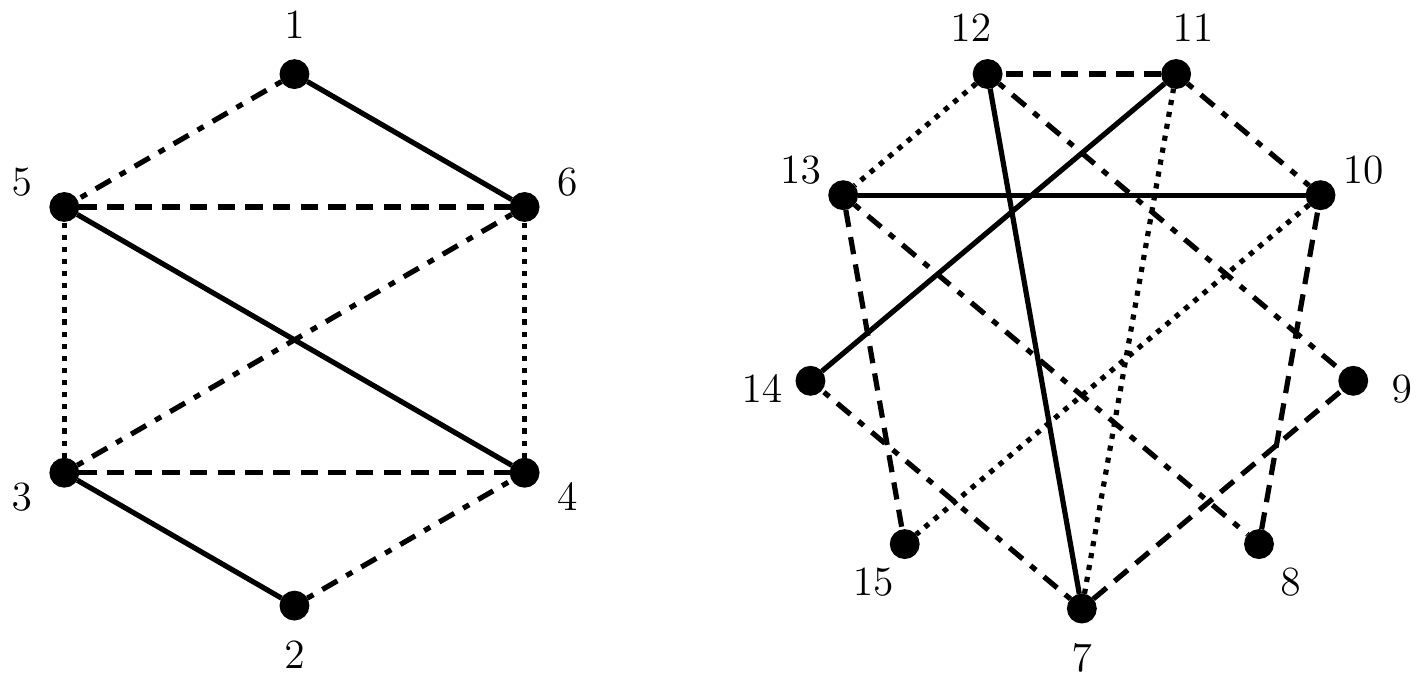}
\caption{Leave of the PSTS$(15)$ given in Example~\ref{E:darryn's example}}\label{F:darrynLeave}
\end{figure}
\end{example}

We now  generalise this small example to find counterexamples to  Conjecture~\ref{C:bryant} (of much larger unspecified order) for all even $w \geq 6$. Our next lemma details how we  generalise the component with six vertices in the leave $L$ in Example~\ref{E:darryn's example}. For an integer $n \geq 2$, let $\mathbb{Z}_n$ denote the additive group of integers modulo $n$.

\begin{lemma} \label{L:L1 colouring}
Let $w\geq 4$ be an even integer and let $L_1$ be the complement of the graph with vertex set $\mathbb{Z}_{w+1} \cup \{\infty\}$ shown in Figure~$\ref{F:L3}$.

Then
\begin{itemize}
    \item[\textup{(i)}]
$\chi'(L_1)=w$; and
    \item[\textup{(ii)}]
in any $w$-edge colouring of $L_1$, the set of colours that hits vertex $1$ equals the set of colours that hits vertex $2$.
\end{itemize}
\end{lemma}

\begin{proof}
A proper $w$-edge colouring $\gamma$ of $L_1$ with colour set $\mathbb{Z}_{w+1}\setminus\{0\}$ is given by
\begin{itemize}
\item $\gamma(xy)=x+y$ for each $xy \in E(L_1)$ with $x,y\in\mathbb{Z}_{w+1}$;
\item $\gamma(x\infty)=2x$ for each $x\in\{2,3,\ldots,w\}$;
\item $\gamma(0\infty)=2$.
\end{itemize}
Thus, since $\Delta(L_1)=w$, we have $\chi'(L_1)=w$ and (i) holds.

To prove (ii), consider an arbitrary $w$-edge colouring of $L_1$.
Since vertices $1$ and $2$ have degree $w-2$ and every other vertex has degree $w$,
there are exactly two colours that miss vertex $1$, exactly two colours that miss vertex $2$,
and each other vertex is hit by every colour. Since any colour that misses a vertex misses at least two vertices, it follows immediately that the two colours that miss vertex $1$ are the same as the two colours that miss vertex $2$.
So (ii) holds.
\end{proof}

\begin{figure}[H]
\centering
\includegraphics[width=0.4\linewidth]{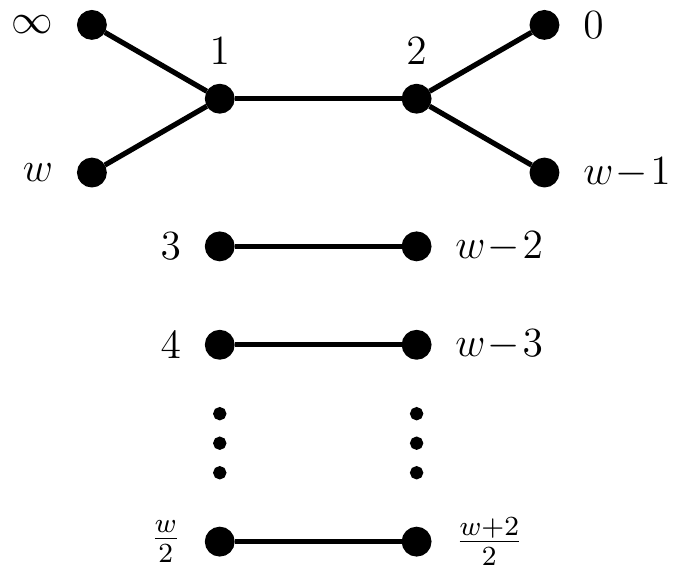}
\caption{The complement of the graph $L_1$ in Lemma~\ref{L:L1 colouring}}
\label{F:L3}
\end{figure}

We also require the following simple consequence of a theorem obtained by Dross \cite{Dross2016} using a result of Barber et al. \cite{BarKuhLoOst}.

\begin{lemma}\label{L:DrossConsequence}
Let $w$ be an even positive integer. There exists an integer $u_0$ such that for any even graph $L$ with odd order $u \geq u_0$, $|E(L)|\equiv \binom{u}{2} \mod{3}$ and $\Delta(L) \leq w$, there is a partial Steiner triple system whose leave is $L$.
\end{lemma}

\begin{proof}
Theorem~7 of \cite{Dross2016} implies that, there exists an integer $n_0$ such that any even graph $G$ with $n \geq n_0$ vertices with $|E(G)| \equiv 0 \mod{3}$ and $\delta(G) \geq \frac{91}{100}n$ is $K_3$-decomposable.
Take $u_0 = \max(n_0,\lceil\frac{100}{9}(w+1)\rceil)$ and suppose that $L$ is an even graph with odd order $u \geq u_0$, $|E(L)|\equiv \binom{u}{2} \mod{3}$ and $\Delta(L) \leq w$. Let $\overline{L}$ be the complement of $L$. It suffices to show that there is a $K_3$-decomposition of $\overline{L}$.

Now, $\overline{L}$ is an even graph because $L$ is an even graph of odd order, and $u \geq u_0 \geq n_0$. Furthermore $\delta(\overline{L}) \geq u-w-1 \geq \frac{91}{100}u$ because $\Delta(L) \leq w$ and $u \geq u_0 \geq \frac{100}{9}(w+1)$. Finally, $|E(\overline{L})| \equiv \binom{u}{2} - |E(L)| \equiv 0 \mod{3}$. Thus we can apply \cite[Theorem~7]{Dross2016} to obtain a \hbox{$K_3$-decomposition} of $\overline{L}$.
\end{proof}

\begin{proof} [\textup{\textbf{Proof of Theorem~\ref{T:counterExample}}}]
A partial Steiner triple system whose leave is a counterexample to Conjecture~\ref{C:bryant} for $w=4$ was exhibited in Example~\ref{E:darryn's example}. Let $w \geq 6$ be an even integer. We will show that there exists a partial Steiner triple system whose leave is a counterexample to Conjecture~\ref{C:bryant} for this value of $w$.

By Lemma~\ref{L:DrossConsequence} there exists an integer $u_0$ such that, for any even graph $L$ with odd order $u \geq u_0$, $|E(L)|\equiv \binom{u}{2} \mod{3}$ and $\Delta(L) \leq w$, there is a partial Steiner triple system whose leave is $L$. Fix an odd $u \geq \max(u_0,4w+1)$ such that $u+w \equiv 1,3 \mod{6}$.
We will find an even graph $L$ of order $u$ such that $|E(L)| \equiv \binom{u}{2} \mod{3}$ and $L$ satisfies conditions (i), (ii) and (iii) of Lemma~\ref{L:counterexampleIdea}. This will suffice to complete the proof because $L$ will have maximum degree at most $w$ by (ii), and so by Lemma~\ref{L:DrossConsequence} there will be a partial Steiner triple system $(U,\mathcal{A})$  with leave $L$.  We will take $L$ to be a vertex-disjoint union of three graphs, $L_1$, $L_2$ and $L_3$, that we now define.

First we let $L_1$ be the graph of order $w+2$ given by Lemma~\ref{L:L1 colouring}. Note that $|E(L_1)| = \binom{w+2}{2} - \frac{w+6}{2} = \frac{1}{2}w^2+w-2$.
Next let  $t=\frac{1}{2}(u-2w-1)$, note that $t \geq w$ because $u \geq 4w+1$, and let $L_{2}$ be the bipartite graph with parts $\{a_0,\ldots,a_{t-1}\}$ and $\{b_0,\ldots,b_{t-1}\}$ and edge set
\[\big\{a_ib_j:i \in \{0,\ldots,t-1\}, j \in \{i,\ldots,i+w-1\}\big\} \setminus \{a_0b_1,a_0b_2,a_1b_1,a_1b_2\}\]
where the subscripts are considered modulo $t$. So $L_2$ is a graph obtained from a $w$-regular bipartite graph of order $2t$ by removing the edges of a $4$-cycle. Hence $|E(L_2)| = \frac{w}{2}(u-2w-1)-4$. Let $L_3$ be the graph with vertex set $\{c_1,\ldots,c_{w-1}\}$ and edge set $\{c_1c_2,c_1c_5,c_2c_5,c_3c_4,c_3c_5,c_4c_5\}$ (note $w \geq 6$). So $L_3$ is the union of the bowtie graph and $w-6$ isolated vertices, and hence $|E(L_3)| =6$.

It only remains to show that $L$ has the properties we desired. Clearly $L$ is an even graph of order $u$.
Now $|E(L)| = \frac{1}{2}w(u-w+1)$ because $|E(L_1)| = \frac{1}{2}w^2+w-2$, $|E(L_2)| = \frac{1}{2}w(u-2w-1)-4$ and $|E(L_3)| = 6$. So $L$ satisfies (i) of Lemma~\ref{L:counterexampleIdea}. Furthermore $|E(L)|\equiv \binom{u}{2} \mod{3}$ because $\binom{u}{2}-|E(L)| = \frac{1}{2}((u+w)(u+w-1)-3uw)$ and $u+w \equiv 1,3 \mod{6}$.
Also, $L$ satisfies (ii) of Lemma~\ref{L:counterexampleIdea} because $\chi'(L_1)=w$ by Lemma~\ref{L:L1 colouring}(i), $\chi'(L_2)=w$ since $L_2$ is bipartite and $\Delta(L_2)=w$, and $\chi'(L_3) \leq w$ since $\Delta(L_3)=4$. Finally, $L$ satisfies (iii) of Lemma~\ref{L:counterexampleIdea} by Lemma~\ref{L:L1 colouring}(ii).
\end{proof}

Theorem~\ref{T:counterExample} shows that the conditions of Conjecture~\ref{C:bryant} do not suffice for the existence of a $K_3$-decomposition of $L \vee K_w$. It remains possible, however, that a slightly strengthened set of conditions does suffice.

\begin{question}
Let $L$ be a graph with $u$ vertices and let $w$ be a nonnegative integer. Do the conditions of Conjecture~\ref{C:bryant} with (4)(iii) replaced by $\Delta(G) \leq w-1$ guarantee the existence of a $K_3$-decomposition of $L \vee K_w$?
\end{question}

Of course, these new conditions are not necessary for the existence of a $K_3$-decomposition of $L \vee K_w$.

\bigskip
\noindent\textbf{Acknowledgments.}\quad
This work was supported by Australian Research Council grants\linebreak DP150100506, DP150100530 and FT160100048.

\end{document}